\documentclass[12pt, reqno]{amsart}
\usepackage[T2A]{fontenc}
\usepackage{inputenc}
\usepackage{amssymb}
\usepackage[matrix,arrow,curve]{xy}
\usepackage{latexsym}
\usepackage{amscd,amsthm,amsfonts,amssymb,amsmath,euscript}

\makeatletter \@addtoreset{equation}{section} \makeatother

\def\Sing{\operatorname{Sing}}
\def \mult {\mathrm{mult}}
\DeclareMathAlphabet{\mathbbold}{U}{bbold}{m}{n}
\def \k {\mathbbold{k}}
\newtheorem{theorem}[equation]{Theorem}
\newtheorem{prop}[equation]{Proposition}
\newtheorem{lemma}[equation]{Lemma}

\theoremstyle{definition}
\newtheorem{example}[equation]{Example}
\newtheorem{definition}[equation]{Definition}

\theoremstyle{remark}
\newtheorem{zam}[equation]{Remark}

\newcommand{\XXX}{{\mathcal{X}_{23}}}
\newcommand{\QQQ}{{\mathcal Q}}
\newcommand{\CCC}{{\mathcal C}}
\newcommand{\CC}{{\mathbb C}}
\newcommand{\BQ}{{D}}

\newcommand{\RR}{{\mathbb R}}

\newcommand{\PP}{{\mathbb P}}
\newcommand{\QQ}{{\mathbb Q}}

\newcommand{\K}{{\mathcal K}}

\newcommand{\GG}{{\mathbb G}}

\newcommand{\tit}{On Unirationality of Quartics over non Algebraically Closed Fields}
\def \ge {\geqslant}
\def \le {\leqslant}

\begin{document}
\begin{title}
\tit
\end{title}
\author{Nikolai F.\,Zak}
\email{nzak@mccme.ru}
\maketitle

\begin{abstract}
We give examples of smooth $\k$-unirational line-free quartic hypersurfaces over a non algebraically closed field
$\k$. Unlike other methods of proving unirationality, our method does not rely on existence of linear spaces on
quartics. 
\end{abstract}

\section{Introduction}
The question of unirationality of algebraic varieties is one of the fundamental and, at the same time, poorly
understood questions of algebraic geometry. It is not even known whether there exists a not unirational Fano
variety. In particular, there is no examples of non unirational quartic threefolds. To the author's knowledge, all
available examples of smooth unirational varieties which are not unirational over a non algebraically closed field
$\k$ are varieties without $\k$-points (cf. [8], [9]).


Let $X$ be an algebraic variety over a field $\k\subset\overline\k$ of characteristic zero, where $\overline\k$ is
the algebraic closure of $\k$. In what follows we denote by $\overline O$ the algebraic closure
$O\times_{\k}\overline{\k}$ of an object $O$ (that is, a field, a variety or a map), defined over $\k$. Unless
specified otherwise, we assume that all algebraic objects we deal with are defined over $\k$. We say that variety
$X$ is smooth if $\overline X$ is smooth. By $\GG(1,n)$ denote the Grassmanian of lines in $\PP^n$.


\begin{definition}\label{def}
A variety $X$ is called {\it $\k$-unirational} (or {\it unirational over $\k$}) if the function field $\k(X)$ is a
subfield of a purely transcendental extension of $\k$ or, in other words, if there exists a dominant rational map
$\PP^N\dashrightarrow X$ defined over $\k$. Variety $X$ is called {\it unirational} if the variety $\overline{X}$ is
unirational over $\overline{\k}$.
\end{definition}

In [1] it is shown that every smooth hypersurface $H_d\subset\PP^n$ of degree $d$ is unirational when $n$ is
sufficiently large. However, the corresponding estimates for $n$ are far from being optimal, so it is important to
develop different ways of proving unirationality. The method used in [1] develops the approach suggested by
U.\,Morin in [5] and is based on the existence (for certain $m$) of an $m$-dimensional linear subspace $L^m\subset
H_d$. Following this method, the authors consider the family of hypersurfaces $H_{d-1}\subset L^{m+1}$ obtained by
intersecting various subspaces $L^{m+1}\supset L^m$ with $H_d$. In section 1.2 of [1] the authors point out that
they do not know of any method of proving unirationality of nonsingular hypersurfaces $H_d$ of degree $d>3$ that
would not use linear subspaces $L\subset H_d$ of positive dimension. In particular, they do not know of any example
of nonsingular $\k$-unirational hypersurface $H_4\subset\PP^n$ without lines. In several recent papers (cf. [3] and
references therein) M.\,Marchisio develops another method of proving unirationality of quartics following an
approach of B.\,Segre in [4]. This method relies on the existence of lines on the quartic, so it does not solve the
problem stated above.

In the present paper we give examples of   nonsingular unirational quartics defined over the field $\RR$ of real
numbers that do not contain real lines. More precisely, we prove the following

\begin{theorem}\label{th}
For every $n\geqslant 8$ there exists an unirational over $\RR$ nonsingular hypersurface $H_4\subset\PP^n$ of degree
four which does not contain any line defined over $\RR$.
\end{theorem}

The structure of the paper is as follows. In Section \ref{0} we prove Proposition \ref{S} that gives a sufficient
condition for $\k$-unirationality of complete intersections of a quadric and a cubic. In Section \ref{1}, using a
birational transformation to the complete intersection of a quadric and a cubic with the required properties, we
prove in Proposition \ref{uniquad} the unirationality over $\k$ of a general singular four-dimensional quartics
that contains a $\k$-rational quadric threefold with the multiplicity two. 
In Section \ref{2} we prove Proposition \ref{uni} on the $\k$-unirationality of general higher-dimensional quartic
containing $\k$-rational three-dimensional quadric with multiplicity two. The method is to reduce this question to
the case of quartics studied in Section \ref{1}. In Section \ref{4} we show the existence of a nonsingular line-free
real quartic that satisfies the conditions of Proposition \ref{uni} and thus complete the proof of Theorem \ref{th}.

\section{$\k$-Unirational Complete Intersection of a Quadric and a Cubic}\label{0}
Unirationality of the nonsingular complete intersection $\XXX\subset\PP^n$ of a quadric and a cubic  for $n\geqslant
5$ is well known, see the sketch of the proof in Example 10.1.3 from [2]. However, the standard proof can not be
applied over a non algebraically closed field because it relies on the family of planes lying on a
higher-dimensional quadric.

In this section we give a sufficient condition for $\k$-unirationality of an irreducible complete intersection of a
quadric and a cubic $\XXX\subset\PP^n$. We assume that $\XXX$ has isolated singularities and does not have singular
points $x$ of multiplicity $\mult_x(\XXX)>2$. Note that in this case the tangent space $T_x(\XXX)$ at an arbitrary
point $x\in\XXX$ is a proper subspace of $\PP^n$ (otherwise both quadric and cubic are singular at $x$ and
$\mult_x(\XXX)\ge 4$).

\begin{zam}
If $\mult_x{\XXX}>2$, $\dim{\XXX}>1$, and $\XXX$ is not a cone, then birational projection $\pi_x:
\XXX\dashrightarrow\PP^{n-1}$ maps $\XXX$ to an irreducible hypersurface $H\subset\PP^{n-1}$ of degree less than
three; hypersurface $H$ is known to be unirational over $\k$ if it is not a cubic cone and has a smooth $\k$-point.
\end{zam}

\begin{prop}\label{S}
Let $\XXX\subset\PP^n$, $n>5$, be the intersection of a nonsingular quadric $\QQQ$ and a cubic $\CCC$ such that
$\XXX$ has isolated singularities and does not have singular points of multiplicity more than two and contains a
$\k$-unirational surface $S\not\subset\Sing{\XXX}$. Suppose $S$ is covered by a family of lines $\mathcal F$ such
that through a general point $s\in S$ there passes the only line $l_s\in\mathcal F$. In this case $\XXX$ is
unirational over $\k$.
\end{prop}

\begin{proof}
Since $S\not\subset\Sing{\XXX}$, the tangent spaces $T_s(\QQQ)$ and $T_s(\CCC)$ at a general point $s\in S$ are
different hyperplanes in $\PP^n$. For such $s$ consider the $(n-3)$-dimensional quadric cone
$$
{\K_s}={\QQQ}\cap T_s(\QQQ)\cap T_s(\CCC)=\QQQ\cap T_s(\XXX).
$$
Let ${\mathcal X}\subset\GG(1,n)\times S$ be the closure of the family of generatrices of the cones $\K_s$. Variety
${\mathcal X}$ is a bundle over $S$ and it's fibers are $(n-4)$-dimensional quadrics $Q_s$ on $\GG(1,n)$ with
respect to Pl\"{u}cker embedding. The lines $l_s$ define the rational section $\varphi: S\dashrightarrow {\mathcal
X}$ of this bundle.

\begin{lemma}\label{ns}
For a general point $s\in S$, the quadric $Q_s$ is irreducible and the point $\varphi(l_s)\in Q_s$ is nonsingular.
\end{lemma}
\begin{proof}
Since $\QQQ$ is nonsingular, the variety ${\QQQ}\cap T_s(\QQQ)$ is a cone over a nonsingular $(n-3)$-dimensional
quadric $Q^{n-3}$. Since $n>5$ quadric $Q_s$ is irreducible as a hyperplane section of $Q^{n-3}$.

Since the quadric $Q^{n-2}=T_s(\QQQ)\cap\QQQ$ is a cone with the only vertex in $s$, the hyperplane section
$\K_s=Q^{n-2}\cap T_s(\CCC)$ can be a cone over a singular quadric only if $T_s(\CCC)$ is tangent to the cone
$Q^{n-2}$ along it's generatrix. Thus, if $Q_s$ is singular in $\varphi(l_s)$, then
$$
l_s\subset\Sing{(T_s(\QQQ)\cap\XXX)}.
$$
If for a general point $s'\in l_s$ the quadric $Q_{s'}$ is singular in $\varphi(l_{s'})=\varphi(l_s)$, then
$l_s\subset\Sing(\XXX)$. If this is true for a general point $s\in S$, then $S\subset\Sing{\XXX}$; this contradicts
our assumptions.
\end{proof}
Combining Lemma \ref{ns} with the following simple fact we see that variety ${\mathcal X}$ is unirational over $\k$.

\begin{lemma}\label{unibund}
Consider a $\k$-unirational variety $S$ and the bundle $f: \mathcal X\to S$ with a general fiber isomorphic to an
irreducible quadric $Q_s$. The variety $\XXX$ is unirational over $\k$ provided the existence of a rational (that
is, defined over an open subset) section $l_s\in Q_s$ such that for a general point $s\in S$ the point
$\varphi(l_s)\in Q_s$ is nonsingular.
\end{lemma}
\begin{proof}
An irreducible quadric $\widetilde{Q_s}$ over the function field $\k(S)$ on the variety $S$ is unirational provided
the existence of a smooth $\k(S)$-point on $\widetilde{Q_s}$. We conclude by considering the inclusions of fields
$\k(\mathcal X)\subset\k(S)(t_1,\ldots,t_a)\subset\k(u_1,\dots,u_b)$ where $t_i$ and $u_i$ are the transcendental
variables and $a$ and $b$ are appropriate dimensions.
\end{proof}

The general generatrix $l\subset{\K_s}$ is tangent to $\CCC$ at $s$, so since $\XXX$ is not a cone (the multiplicity
of the vertex of such a cone would be equal to six), $l$ intersects ${\XXX}$ in exactly one other point $s_l$.
Consider rational map $f:\mathcal X\dashrightarrow\XXX$ such that $f(l)=s_l$. To prove Proposition \ref{S} it is
sufficient to check that $f$ is dominant. Due to $\k$-unirationality of ${\mathcal X}$, the generatrices of cones
$\K_s$ are dense in the set of generatrices of cones $\overline{\K_s}$, so it is sufficient to check the dominance
of $\overline{f}$ which follows from

\begin{lemma}\label{lll}
The variety $\overline{\QQQ'}$ sweeped out by cones $\overline{\K_s}\subset\overline{\QQQ}$ with $s\in\overline{S}$
coincides with the quadric $\overline{\QQQ}$.
\end{lemma}

\begin{proof}
If $\dim{\overline{\QQQ'}}\le n-3$ then for two general points $s, s'\in\overline S$ we have
$\overline{\K_{s}}=\overline{\K_{s'}}$ and thus $S\subset\Sing{\K_s}$ which contradicts Lemma \ref{ns}. Assume that
$\dim{\overline{\QQQ'}}=n-2$. Since $\overline{\QQQ'}$ contains a $(2\cdot(n-2)-3)$-dimensional family of lines,
$\overline{\QQQ'}$ is either a quadric or a one-dimensional family of the projective spaces (the fact proven in [6]
states that any $n$-dimensional projective variety containing a $2n-3$-dimensional family of lines is either a
quadric or a one-dimensional family of the projective spaces). The second case is impossible because for $n>5$ a
nonsingular quadric $\QQQ\subset\PP^n$ does not contain projective subspaces of dimension $n-3$.

In the first case, for an arbitrary point $s\in\overline{S}$ we have
$\overline{\K_s}=T_s(\overline{\QQQ'})\cap\overline{\QQQ'}$, so that
$$
\overline{S}\subset\Sing{(\overline{\CCC}\cap\overline{\QQQ'})}.
$$
Note that for $x\in\overline{S}\cap\Sing{\overline{\XXX}}$ we have $\overline{\QQQ'}\subset T_x(\overline \XXX)$. In
fact, over the field $\overline{k}=\CC$ it follows from the continuity considerations and over other fields we can
apply Lefschetz principle. Thus, the $(n-1)$-dimensional projective span of $\overline{\QQQ'}$ is tangent to the
$(n-2)$-dimensional variety $\overline{\XXX}$ along a surface $\overline{S}$ which contradicts F.\, L.\, Zak's
Theorem on Tangencies (see Corollary 1.8 in [10]). Thus, $\dim{\overline{\QQQ'}}=n-1$ and so
$\overline{\QQQ'}=\overline{\QQQ}$.
\end{proof}

Being the image of $\k$-unirational variety $\mathcal X$ under the dominant rational map, $\XXX$ is unirational over
$\k$ and Proposition \ref{S} is proved.
\end{proof}

\section{$\k$-Unirational Four-Dimensional Quartics}\label{1}
The main goal of this section is to prove the following

\begin{prop}\label{uniquad}
Let $Q$ be a three-dimensional nonsingular $\k$-rational quadric\footnote{A quadric is $\k$-rational if and only if
it is irreducible and has a smooth point.}. Then a general quartic $Y_4\subset\PP^5$ containing $Q$ with
multiplicity two is unirational over $\k$.
\end{prop}
\begin{proof}
Below we prove an auxiliary result, which holds in a more general framework for quartics of arbitrary dimension.
Consider an irreducible normal quartic $Y_4\subset\PP^{n-1}$, $n>3$, containing (not necessarily with multiplicity
two) an $(n-3)$-dimensional smooth quadric $Q$ (which is not necessarily rational over $\k$).

Consider a cone $\K_{Y_4}\subset\PP^n$ over $Y_4$ with vertex at a general point $x\in\PP^n$. Let
$\K_Q\subset\K_{Y_4}$ be a cone over the quadric $Q$ with the same vertex. Let $\QQQ\subset\PP^n$ be a nonsingular
quadric containing $\K_Q$. The following proposition reduces the question of $\k$-unirationality of $Y_4$ to the
question of $\k$-unirationality of a complete intersection of a quadric and a cubic $\XXX\subset\PP^n$.

\begin{prop}\label{proj}
$Y_4=\pi_x(\XXX)$, where ${\XXX}=\QQQ\cap\CCC_{\QQQ}$ is a complete intersection of $\QQQ$ and a cubic
${\CCC_\QQQ}\subset\PP^n$ and $\pi_x$ is the birational projection from the singular point $x\in\XXX$.
\end{prop}

\begin{proof}
Consider the intersection $\mathcal Q\cap \K_{Y_4}$. Since $\deg(\overline{\mathcal Q\cap \K_{Y_4}})=8$, we have
(scheme-theoretically)
$$
{\mathcal Q}\cap \K_{Y_4}=\K_Q\cup {\XXX},
$$
where $\overline{{\XXX}}\subset\overline{{\mathcal Q}}$ is a divisor of degree six. By Lefschetz theorem, variety
$\overline{\XXX}$ is an intersection of the quadric $\overline{\mathcal Q}$ with a cubic\footnote{Such a cubic is
not unique but our choice does not affect the exposition.} $\overline{\CCC_\QQQ}\subset\PP^n$. Since
$\overline{\XXX}$ is defined over $\k$, it is easy to show that we can choose cubic $\overline{\CCC_\QQQ}$ to be
also defined over $\k$ since the corresponding ideal has to be generated by a quadric and a cubic over $\k$. Note
that $\XXX$ is irreducible. In fact, otherwise it would be a union of a quadric and an intersection of two quadrics
(other possible cases are not allowed by Lefschetz theorem), which contradicts the normality of $Y_4$. Thus, the
projection $\pi_x:\XXX\dashrightarrow Y_4$ from the point $x$ is a surjective map of degree one and $x\in\XXX$ is a
double point.
\end{proof}

Consider now a general quartic $Y_4\subset\PP^5$ containing $\k$-rational quadric threefold $Q$ with multiplicity
two. Note that such $Y_4$ is normal since it's singular locus has codimension two by Lemma \ref{l}. Therefore,
taking into account Proposition \ref{proj}, to finish the proof of Proposition \ref{uniquad} it suffices to find
such a nonsingular quadric $\QQQ\supset\K_Q$, that the corresponding variety $\XXX$ is unirational over $\k$. Note
that $T_x(\XXX)=T_x(\QQQ)$ because $\mult_x(\XXX)=2$ and $\QQQ$ is nonsingular. Thus, we have
$$
T_x(\XXX)\cap\XXX=\K_{Q}\cap\XXX.
$$
Note that in our situation the quadric $Q$, being the image of the exceptional divisor of the blow up of the point
$x\in\XXX$, lies in the base of the tangent cone $\K_x(\XXX)$, so $\K_x(\XXX)=\K_Q$. Thus, if the cubic
$\CCC_{\QQQ}$ contains a nonsingular point $s\in\K_Q$ then $\CCC_{\QQQ}$ intersects a line joining $x$ and $s$ with
multiplicity four and therefore $\XXX$ contains this line.

Let $C\subset Q$ be a $\k$-rational curve. If for some nonsingular quadric $\QQQ\supset\K_Q$ the cubic $\CCC_\QQQ$
contains $C$, then, by the above speculation, the corresponding variety $\XXX$ contains the ($\k$-rational) cone $S$
over $C$ with the vertex in $x$ and variety $\XXX$ is unirational over $\k$ by Proposition \ref{S} if
$S\not\subset\Sing{\XXX}$. This condition holds for a general quartic $Y_4\subset\PP^5$ containing a $\k$-rational
nonsingular quadric threefold $Q$ with multiplicity two. In fact, general quartic of this type has only ordinary
double singularities, so a line $l\subset\Sing{\XXX}$ can not be contracted to such a point.

\begin{lemma}
For every conic $C\subset Q$ there exists a nonsingular quadric $\QQQ\supset\K_Q$ such that cubic $\CCC_\QQQ$
contains $C$.
\end{lemma}
\begin{proof}
It is not hard to compute that the dimension of the projectivised linear system of quadrics $\QQQ\subset\PP^6$
containing $\K_\QQQ$ equals to seven. The cubic $\CCC_\QQQ$ can correspond to not more than two quadrics since
$\QQQ\cap\CCC_\QQQ\subset\CCC_\QQQ\cap\K_{Y_4}$ and the intersections of the cubic with different quadrics are
different. Therefore, the dimension of the projectivised linear system of cubics $\CCC_\QQQ$ on $\K_{Y_4}$ also
equals to seven. The space of cubics containing a given conic $C$ has codimension seven in the space of all cubics,
so there exists such a quadric $\QQQ$ that $\CCC_{\QQQ}\supset C$.

If $\QQQ$ is singular then $\QQQ$ is a cone with vertex in $x$ over a quadric $Q_1$ such that $\PP^5_{Y_4}\supset
Q_1\supset Q$, where $\PP^5_{Y_4}$ is the projective span of $Y_4$. It is easy to show that there exists a
one-dimensional family with an open base of nonsingular five-dimensional quadrics $\QQQ'\supset(\K_Q\cup Q_1)$. Note
that the restrictions $R_\QQQ$ and $R_{\QQQ'}$ of cubics $\CCC_\QQQ$ and $\CCC_{\QQQ'}$ to the projective span of
quartic $Y_4$ are the same since they are given by the same equation $Q_1\cap Y_4=Q\cup R$. Therefore,
$\CCC_{\QQQ'}\cap\QQQ_1=\CCC_\QQQ\cap\QQQ_1$ and thus $\CCC_{\QQQ'}\supset C$ and we are done.
\end{proof}

Let $C$ be a $\k$-rational conic on $Q$ (such conics exist because $Q$ is rational over $\k$) and $\XXX$ be an
intersection $\QQQ\cap\CCC_\QQQ$ where $\CCC_{\QQQ}\supset C$. Due to above speculation, $\XXX$ is unirational over
$\k$. Since $Y_4$ is birational to $\XXX$, Proposition \ref{uniquad} is proved.
\end{proof}

\section{$\k$-Unirational Higher-Dimensional Quartics} \label{2}


\begin{prop}\label{uni}
A general nonsingular quartic $H_4\subset\PP^n$ intersecting a four-dimensional space $M\subset\PP^n$ by a
$\k$-unirational quadric threefold $Q$ with multiplicity two is unirational over $\k$.
\end{prop}

\begin{proof}
Let $\Psi_4^n$ be the universal family of four-dimensional quartics in $\PP^n$, that is the set of pairs
$\{(Y_4,y)\}$, where $Y_4\subset\PP^n$ is a four-dimensional quartic and $y\in Y_4$ is a point. Consider the family
$\Lambda\subset\Psi_4^n$ with base $B=\PP^{n-5}$ of various sections $H_4\cap G$  by five-dimensional spaces
$G\supset M$.

Note that a general quartic $Y_4$ from $\Lambda$ is irreducible. Otherwise, since $Y_4$ does not contain $M$ it
would be a union of two quadrics: $Y_4=\mathcal{Q}_1\cup\mathcal{Q}_2$. Thus,
$\mathcal{Q}_1\cap\mathcal{Q}_2\subset\Sing{Y_4}$, so $\dim{\Sing{Y_4}}\geqslant 3$ and by Bertini's theorem
$\Sing{Y_4}=Q$. Therefore, $\Sing{H_4}\supset Q$, which contradicts the assumption that $H_4$ is nonsingular.

Since $\Lambda\supset B\times Q$, the four-dimensional quartic $\widetilde{Y_4}$ over the function field $\k(B)$
corresponding to the general fiber of $\Lambda$ contains a nonsingular three-dimensional quadric $\widetilde Q$
defined over $\k(B)$, with multiplicity two and does not contain its projective span $\widetilde M=\PP^4_{\k(B)}$.
For each $\k$-point $q\in Q$ holds $\Lambda\supset B\times q$, so the quadric $\widetilde Q$ contains
$\k(B)$-points,  so it is unirational over $\k(B)$. Thus, quartic $\widetilde{Y_4}$ is unirational over $\k(B)$ by
Proposition \ref{uniquad}, so $\Lambda$ is unirational over $\k$ (see the similar reasoning in the proof of Lemma
\ref{unibund}).  Since $H_4$ is birational to $\Lambda$ (they are isomorphic outside $M\cap H_4$), Proposition
\ref{uni} is proved.
\end{proof}

\begin{zam}
It can be proved (see [1] for the general result about the maximal variation of the linear sections of hypersurfaces
of low degree in the corresponding moduli space) that for a big $n$ general quartic in $\PP^n$ over an algebraically
closed field contains a smooth quadric threefold with multiplicity two and does not contain its projective span.
This observation gives an alternative proof of unirationality of higher-dimensional quartics (see other proofs in
[1], [3], [4], [5]).
\end{zam}

\section{Examples of nonsingular $\k$-Unirational Quartics without Lines}\label{4}
Consider the field $\RR$ of real numbers. Let $Q$ be a real sphere
$$
Q\subset M\subset\PP^n,\qquad  n\geqslant 8,
$$
given by equation
$$
f=x_0^2+x_1^2+x_2^2+x_3^2-x_4^2=0
$$
in homogeneous coordinates $(x_0:\ldots:x_4)$ on $M=\PP^4$. Denote by $\BQ$ the double quadric $Q$, that is, the
quartic given by equation $f^2=0$ in $M$. Consider a general hyperplane $\Gamma\subset\PP^n$ which does not
intersect $\BQ$. Note that the set of $(n-1)$-dimensional real quartics containing $\BQ$ and not intersecting
$\Gamma$ is open in the standard analytic real topology and is not empty since it contains, for example (for a far
enough hyperplane $\Gamma$) a quartic given by equation $f^2+x_4^4+\dots+x_n^4=0$. Since $\dim{\Sing{\BQ}}=3$,
general (in Zariski topology) $(n-1)$-dimensional real quartic containing $\BQ$ is nonsingular by the following
simple fact, the proof of which is straightforward.

\begin{lemma}\label{l}
Let $H^{N-1}_d\subset\PP^N$ be a hypersurface of degree $d$ over $\overline\k$ such that
$$
\dim{\Sing{H^{N-1}_d}}=m.
$$
Then for the general hypersurface $H_d^{N+k-1}\supset H_d^{N-1}$ we have
$$
\dim\Sing{H_d^{N+k-1}}=m-k.
$$

\end{lemma}

Note that any non empty set open in analytic topology always meets an open non empty set in Zariski topology, so
there exists a nonsingular quartic $H_4\subset\PP^n$ which contain $\BQ$, satisfy the conditions of Proposition
\ref{uni} and does not intersect hyperplane $\Gamma$. In particular, $H_4$ does not contain any real lines and
yields the required example, so Theorem \ref{th} is proved.

\begin{zam}
In the examples proving Theorem \ref{th} we use two properties of field $\RR$. The first one is the existence of an
element $r\in\RR$ which can not be decomposed into a sum of squares of real numbers (this property is necessary for
the construction of $\k$-rational quadric $Q$ which does not contain lines), and the second is the fact that a non
empty set open in the analytic topology always intersects a non empty set open in Zariski topology. Such fields as
$\RR(t)$ have the same properties and similar examples are valid for them. Also, since rational points are dense in
an open set in analytic topology, the statement of Theorem \ref{th} also holds for field $\QQ$ as well as for
$\QQ(t)$.
\end{zam}

\medskip

The author is grateful to S.\,Galkin, S.\,Gorchinskiy, V.\,Iskovskikh, Yu.\,Prokhorov, K.\,Shra\-mov and F.\,Zak for
useful discussions.

\bigskip

\end{document}